\documentclass{amsart}
\usepackage{amsmath, amsthm, amscd, amsfonts, amssymb}
 \newtheorem{thm}{Theorem}[section]
 \newtheorem{cor}[thm]{Corollary}
 \newtheorem{lem}[thm]{Lemma}
 
 \theoremstyle{definition}

 \theoremstyle{remark}

 \numberwithin{equation}{section}

\begin{document}

\title[Remarks on the tensor degree of finite groups] {Remarks on the tensor degree of finite groups}

\author[A.M.A. Alghamdi]{Ahmad M.A. Alghamdi}
\address{Department of Mathematical Sciences, Faculty of Applied Sciences\\
Umm Alqura University, P.O. Box 14035, Makkah, 21955, Saudi Arabia}
\email{amghamdi@uqu.edu.sa}

\author[F.G. Russo]{Francesco G. Russo}
\address{DEIM, Universit\'a degli Studi di Palermo\\
Viale Delle Scienze, Edificio 8, 90128, Palermo, Italy}
\email{francescog.russo@yahoo.com}


\date \today


\maketitle

\begin{abstract} The present paper is a note on the tensor  degree  of  finite groups,  introduced recently in literature. This  numerical invariant generalizes the  commutativity degree through the notion of nonabelian tensor square. We show two inequalities,  which correlate the tensor and the commutativity degree of finite groups, and, indirectly,  structural properties will be discussed. 
\end{abstract}

\section{The relative tensor degree}

All the groups of the present paper are supposed to be finite.  Having in mind the exponential notation for the conjugation of two elements $x$ and $y$ in a group $G$, that is, the notation $x^y=y^{-1}xy$, we may follow  \cite{bjr, brownbook, ort} in saying that  two normal subgroups $H$ and $K$ of  $G$  \textit{act compatibly} upon each other, if
\[\left(h_2^{k_1}\right)^{h_1} = \left(\left({h_2}^{h_1^{-1}}\right)^{k_1}\right)^{h_1} \ \mathrm{and} \ \left(k_2^{h_1}\right)^{k_1} = \left(\left({k_2}^{k_1^{-1}}\right)^{h_1}\right)^{k_1}\]
for all $h_1,h_2 \in H$ and $k_1,k_2 \in K$, and if $H$ and $K$ act upon themselves
by conjugation. Given $h \in H$ and $k \in K$, the  nonabelian tensor product $H \otimes K$  is  the group generated by the symbols $h \otimes k$ satisfying the relations $h_1h_2\otimes k_1=(h_2^{h_1}\otimes {k_1}^{h_1}) \ (h_1\otimes k_1)$ and $k_1k_2\otimes h_1=(k_1\otimes h_1) \ (h_1^{k_1} \otimes k_2^{k_1})$ for all  $h_1, h_2 \in H$ and $k_1, k_2 \in K$. The map \[\kappa_{H,K} : h\otimes k \in H \otimes K \mapsto [h,k] \in [H,K]\] turns out to be an epimorphism, whose kernel $\ker \kappa_{H,K} =J(G,H,K)$ is central in $H \otimes K$. The reader may find more details and a topological approach to $J(G,H,K)$ in \cite{brownbook, ellis, prf, ort}. The short exact sequence  \[\begin{CD}
1 @>>>J(G,H,K) @>>> H \otimes K @>\kappa_{H,K}>> [H,K]@>>>1 \end{CD}\] is a central extension. In the special case $G=H=K$, we have that  $J(G)=J(G,G,G)=\ker \kappa_{G,G}= \ker \kappa$ and $H \otimes K = G \otimes G$ is called nonabelian tensor square of $G$. The fundamental properties of $G \otimes G$ have been described in the classical paper  \cite{bjr},  in which it is noted that $\kappa : x\otimes y \in G\otimes G \mapsto \kappa(x\otimes y)=[x,y] \in G'$
is an epimorphism of groups with  $\ker \kappa =J(G)$ and $1 \rightarrow J(G) \rightarrow G \otimes G {\overset{\kappa}\rightarrow} G' \rightarrow 1$ is a central extension. The group  $J(G)$ is important from the perspective of the algebraic topology, in fact  $J(G)\cong \pi_3(SK(G,1))$ is the third homotopy group  of the suspension of an Eilenberg--MacLane space $K(G,1)$ (see \cite{brownbook} for more details).

As done in \cite{pf}, we may consider  the \textit{tensor centralizer}  \[C^\otimes_K(H)=\{k \in K \ | \ h \otimes k=1,  \  \forall h \in H\}=\bigcap_{h \in H} C^\otimes_K(h)\] and the \textit{tensor center} $C^\otimes_G(G)=Z^\otimes(G)={\underset{x \in G}\bigcap} C_G^\otimes(x)$ and one can check that $C^\otimes_G(x)$ and $Z^\otimes(G)$ are subgroups of $G$ such that $C^\otimes_G(x) \subseteq C_G(x)$ and $Z^\otimes(G) \subseteq Z(G)$.

Generalizing what has been done in \cite{pf}, we  may define the \textit{relative tensor degree}  \[d^\otimes(H,K)=\frac{|\{(h,k)\in H \times K \ | \  h \otimes k =1\}|}{|H||K|}= \frac{1}{|H| \ |K|} \sum_{h \in H} |C^\otimes_K(h)|\] of $H$ and $K$.  Notice that  $d^\otimes(G)=1$ if and only if $Z^\otimes(G)=G$. Unfortunately, few results  are available on the relative tensor degree at the moment and these are contained mainly in \cite{pf}, where it is discussed the \textit{tensor degree} $d^\otimes(G)=d^\otimes(G,G)$ of $G$. On the other hand, there is a rich literature  (see for instance \cite{alghamdi, amit, gustafson, l1, l2}) on   the \textit{relative commutativity degree} \[d(H,K)=\frac{|\{(h,k)\in H \times K \ | \  [h,k] =1\}|}{|H||K|}= \frac{1}{|H| \ |K|} \sum_{h \in H} |C_K(h)|=\frac{k_K(H)}{|H|}\]  of $H$ and $K$ (not necessarily normal this time) of $G$.  Here $k_K(H)$ is the number of $K$--conjugacy classes that constitute $H$.   In particular, if $G=H=K$,  we find the well known \textit{commutativity degree} $d(G)=d(G,G)=k_G(G)/|G|$.   Our fist main contribution is the following.

\begin{thm}\label{boundrelative}Let $H,K$ be two normal subgroups of a group $G$ and $p$  the smallest prime divisor of $|G|$. Then the following inequalities are true:
\[\frac{d(H,K)}{|J(G,H,K)|}+\frac{|C^\otimes_K(H)|}{|H|} \left(1-\frac{1}{|J(G,H,K)|}\right) \le d^\otimes(H,K) \leqno{(a)}\]
\[d^\otimes(H,K)\le d(H,K)-\left(1-\frac{1}{p}\right)\left(\frac{|C_K(H)|-|C^\otimes_K(H)|}{|H|}\right). \leqno{(b)}\]
\end{thm}

On the other hand,  we may  correlate the relative tensor degree, the relative commutativity degree and another notion, studied recently in \cite{prf}. In order to proceed in this direction, we recall from \cite{bjr, ellis, niru} that the nonabelian exterior product  $H \wedge K$ of $H$ and $K$ is the quotient the nonabelian tensor produc $H \otimes K$, defined by $H \wedge K= (H \otimes K)/ \nabla(H\cap K)=\langle (x \otimes y) \nabla(H \cap K) \ | \ x,y \in H \cap K \rangle=\langle x \wedge y \ | \ x,y \in H \cap K \rangle$, where $\nabla(H \cap K)= \langle x \otimes  x | \ x \in H \cap K\rangle$. From \cite{bjr, brownbook},  we may note that   \[\kappa'_{H,K} : h\wedge k \in H\wedge K \mapsto \kappa'_{H,K}(h\wedge k)=[h,k] \in [H,K]\] is an epimorphism of groups such that \[\begin{CD}
1 @>>>M(G,H,K) @>>> H \wedge K @>\kappa'_{H,K}>> [H,K]@>>>1 \end{CD}\]  is a central extension,  where $M(G,H,K)= \ker \kappa'_{H,K}$ is the so--called \textit{Schur multiplier of the triple} $(G,H,K)$.  We inform the reader that several references on the theory of the Schur multipliers of triples can be found in \cite{brownbook,  prf}. In  particular, $ M(G,G,G)=M(G)=H_2(G,\mathbb{Z})$ is the  \textit{Schur multiplier} of $G$, that is, the second integral homology group over $G$.  

In our situation, it is possible to consider the set 
\[C^\wedge_K(H)=\{k \in K \ | \ h \wedge k=1,  \  \forall h \in H\}=\bigcap_{h \in H} C^\wedge_K(h),\] called  \textit{exterior centralizer} of $H$ with respect to $K$ and it is actually a subgroup of $K$ (see  \cite{niru} for details). In particular, $C^\wedge_G(G)=Z^\wedge(G)=\bigcap_{x \in G} C^\wedge_G(x) $ is called  \textit{exterior center} of  $G$. It is easy to  check  that $C^\wedge_G(x) \subseteq C_G(x)$ and $Z^\wedge(G) \subseteq Z(G)$.

Some recent papers as \cite{prf} show  that it is possible to have a combinatorial approach for measuring how far a group $G$ is from $Z^\wedge(G)$ and  this is interesting, because  a result of Ellis \cite{ellis}  characterize a capable groups by the triviality of its exterior center  (i.e.: a group $G$ is  \textit{capable} if $G\simeq E/Z(E)$ for a given group $E$).    This aspect has motivated the notion of \textit{relative exterior degree}  \[d^\wedge(H,K)=\frac{|\{(h,k)\in H \times K \ | \  h \wedge k =1\}|}{|H||K|}= \frac{1}{|H| \ |K|} \sum_{h \in H} |C^\wedge_K(h)|\]  of $H$ and $K$. When $G=H=K$, we find the \textit{exterior degree} $d^\wedge(G,G)=d^\wedge(G)$ of $G$ in \cite{prf}. It is  easy to prove  that $d^\wedge(G)=1$ if and only if $G=Z^\wedge(G)$. Hence the exterior degree represents the probability that two randomly chosen elements commute with respect to the operator $\wedge$. Roughly speaking, this means that there are many chances of finding  capable groups  for small values of  exterior degree.

From \cite[Theorem 2.8]{pf}, we may correlate the above notions via the inequality
\[d^\otimes(G) \le d^\wedge(G) \le d(G)\]
and our second main theorem shows that something of more general holds.

\begin{thm}\label{inequalities}
Let $H,K$ be normal subgroups of a group $G$. Then   \[d^\otimes(H,K) \le d^\wedge(H,K) \le d(H,K).\]
Moreover, if $J(G,H,K)$ is trivial, then $d^\otimes(H,K)=d^\wedge(H,K)=d(H,K)$.
\end{thm}

\section{Proofs of the main results}

We begin with a technical lemma, whose proof uses an argument which appears in \cite[Lemma 2.1]{prf} and \cite[Lemma 2.2]{pf}  in  different ways.

\begin{lem} \label{technicallemma}
Let $H,K$ be  normal subgroups of a group $G$. Then
\[d^\otimes (H,K)=  \frac{1}{|H|} \sum^{k_K(H)}_{i=1}  \frac{|C^\otimes_K(h_i)|}{|C_K(h_i)|}.\]
In particular, if $G=HK$, then $C_K(h_i)/C^\otimes_K(h_i)$ is isomorphic to a subgroup of $J(G,H,K)$ and $|C_G(h_i):C^\otimes_G(h_i)| \le |J(G,H,K)|$ for all $i=1,2, \ldots, k_K(H)$.
\end{lem}

\begin{proof} Since $H$ is normal in $G$, we  consider  the $K$--conjugacy classes $C_1, \ldots, C_{k_K(H)}$ that constitute $H$. It follows that
\[|H| \ |K| \ d^\otimes (H,K)=  \sum_{h\in H} |C^\otimes_K(h)|=\sum^{k_K(H)}_{i=1} \sum_{h \in C_i} |C^\otimes_K(h)|\]
\[=\sum^{k_K(H)}_{i=1} |K:C_K(h_i)| \ |C^\otimes_K(h_i)|=|K| \ \sum^{k_K(H)}_{i=1}   \frac{|C^\otimes_K(h_i)|}{|C_K(h_i)|}.\]
Now assume that $G=HK$. For all $i=1,\ldots, k_K(H)$,
the map   \[\varphi: kC^\otimes_K(h_i) \in C_K(h_i)/C^\otimes_K(h_i) \longmapsto k \otimes h_i \in J(G,H,K)\]
satisfies the condition \[\varphi (k_1k_2C^\otimes_K(h_i))=k_1k_2 \otimes h_i = (k_1 \otimes h_i)^{k_2} \ (k_2 \otimes h_i)\]\[=(k_1 \otimes h_i) \ (k_2 \otimes h_i)=\varphi(k_1C^\otimes_K(h_i)) \ \varphi(k_2C^\otimes_K(h_i))\] for all $k_1,k_2 \in C_K(h_i)$. Furthermore,  $\ker \varphi = \{kC^\otimes_K(h_i) \ | \ k \otimes h_i = 1\}=C^\otimes_K(h_i)$. Then $\varphi$ is a monomorphism and  $C_K(h_i)/C^\otimes_K(h_i)$ is isomorphic to a subgroup of $J(G,H,K)$. We conclude that $|C_K(h_i):C^\otimes_K(h_i)| \le |J(G,H,K)|$.
\end{proof}

Now we may prove Theorem \ref{boundrelative}. It is an interesting bound, which connects the notion of relative tensor degree with that of relative commutativity degree.

\begin{proof}[Proof of Theorem \ref{boundrelative}]  We begin to prove $(a)$. From Lemma \ref{technicallemma}, we have \[|C^\otimes_K(H)|/|C_K(H)| \ge 1/|J(G,H,K)|\] and, together with the equality $d(H,K)=\frac{k_K(H)}{|H|}$, we deduce
\[d^\otimes(H,K)=\frac{1}{|H|} \sum^{k_K(H)}_{i=1}\left| \frac{C^\otimes_K(h_i)}{C_K(h_i)}\right|\ge \frac{1}{|H|} \left( |C^\otimes_K(H)|+\frac{k_K(H)-|C^\otimes_K(H)|}{|J(G,H,K)|}\right)\]
\[=\frac{k_K(H)}{|H| \ |J(G,H,K)|}+\frac{|C^\otimes_K(H)|}{|G|}\left(1-\frac{1}{ |J(G,H,K)|}\right)\]
\[=\frac{d(H,K)}{|J(G,H,K)|}+\frac{|C^\otimes_K(H)|}{|H|} \left(1-\frac{1}{|J(G,H,K)|}\right)\]

Conversely, $|K:C^\otimes_K(h_i)| \ge p$ implies that
\[d^\otimes(H,K)=\frac{1}{|H|} \sum^{k_K(H)}_{i=1}\left| \frac{C^\otimes_K(h_i)}{C_K(h_i)}\right|\]
\[\le \frac{|C^\otimes_K(H)|}{|H|}+ \frac{1}{p} \left(\frac{|C_K(H)|-|C^\otimes_K(H)|}{|H|}\right)+\frac{k_K(H)-|C_K(H)|}{|H|}\]
\[=d(H,K)-\frac{p-1}{p}\left(\frac{|C_K(H)|-|C^\otimes_K(H)|}{|H|}\right).\]
\end{proof}

Immediately, we note that \cite[Theorem 2.3]{pf} describes a special case of Theorem \ref{boundrelative}. The following consequence of Theorem \ref{boundrelative} is  interesting, too.

\begin{cor}\label{conseq}Let $G=HK$ for two normal subgroups $H$ and $K$. Then
\[ \frac{d(H,K)}{|J(G,H,K)|}\le  d^\otimes(H,K) \le   \ d(H,K).\]
In particular, if $J(G,H,K)$ is trivial, then $d^\otimes(H,K)=d(H,K).$
\end{cor}

The second main theorem is a result of comparison. Its proof is the following.

\begin{proof}[Proof of Theorem \ref{inequalities}]
We have  \[d^\wedge(H,K) = \frac{1}{|H|} \sum^{k_K(H)}_{i=1}\left| \frac{C^\wedge_K(h_i)}{C_K(h_i)} \right| \le \frac{1}{|H|} \sum^{k_K(H)}_{i=1}\left| \frac{C_K(h_i)}{C_K(h_i)} \right| =d(H,K)\]
and the upper bound follows.

 Now  $k \in C^\wedge_K(H)$ if and only if $k \wedge h = 1$ for all $h \in H$ if and only if $(k \otimes h) \nabla(H \cap K)=\nabla(H \cap K)$ if and only if $k \otimes h \in \nabla(H \cap K)$. This condition is weaker than the condition $k \otimes h = 1$, characterizing the elements of $ C^\otimes_K(H)$. Then $ C^\otimes_K(H)\subseteq C^\wedge_K(H) \subseteq C_K(H)$. This and  Lemma \ref{technicallemma} imply the lower bound
\[d^\otimes(H,K) = \frac{1}{|H|} \sum^{k_K(H)}_{i=1}\left| \frac{C^\otimes_K(h_i)}{C_K(h_i)} \right| \le \frac{1}{|H|} \sum^{k_K(H)}_{i=1}\left| \frac{C^\wedge_K(h_i)}{C_K(h_i)} \right| =d^\wedge(H,K).\]
The rest follows from Corollary \ref{conseq}.
\end{proof}

\end{document}